\documentclass[11pt,a4paper]{article}

\usepackage[english]{babel}
\usepackage[utf8]{inputenc}
\usepackage[T1]{fontenc}
\usepackage{lmodern}
\usepackage{amsmath}
\usepackage{amssymb}
\usepackage{mathrsfs}
\usepackage{amsthm}
\usepackage[shortlabels]{enumitem}
\usepackage{xcolor}

\usepackage{hyperref}  % pour les hyperliens (table des matières notamment)
\hypersetup{                    % parametrage des hyperliens
	colorlinks=true,                % colorise les liens
	breaklinks=true,                % permet les retours à la ligne pour les liens trop longs
	urlcolor= blue,                 % couleur des hyperliens
	linkcolor= black,                % couleur des liens internes aux documents (index, figures, tableaux, equations,...)
	citecolor= magenta,                % couleur des liens vers les references bibliographiques
	pdfstartview = FitH,
	bookmarksopen = true               % marques-page déroulés
}

\usepackage{dsfont}

\newtheorem{theorem}{Theorem}[section]
\newtheorem{lemma}[theorem]{Lemma} 
\newtheorem{corollary}[theorem]{Corollary}
\newtheorem{proposition}[theorem]{Proposition}
\newtheorem{remark}[theorem]{Remark}

\newcommand{\R}{\mathbb{R}}
\newcommand{\N}{\mathbb{N}}

\newcommand{\Z}{\mathbb{Z}}

\author{Mathias Beiglb\"ock, George Lowther, Gudmund Pammer, Walter Schachermayer}

\title{Faking Brownian motion with continuous Markov martingales}
\begin{document}

\maketitle

\begin{abstract}
	Hamza-Klebaner \cite{HaKl07}  posed the problem of constructing martingales with Brownian mar\-ginals that differ from Brownian motion, so called \emph{fake Brownian motions}. Besides its theoretical appeal, the problem represents the quintessential version of the ubiquitous fitting problem in mathematical finance where the task is to construct martingales that satisfy marginal constraints imposed by market data.
 
	Non-continuous solutions to this challenge were given by Madan-Yor \cite{MaYo02}, Hamza-Kle\-baner \cite{HaKl07}, Hobson \cite{Ho13}, and Fan-Hamza-Klebaner \cite{FaHaKl15} whereas continuous (but non-Markovian) fake Brownian motions were constructed by Oleszkiewicz \cite{Ol08}, Albin \cite{Al08}, Baker-Donati-Yor \cite{BaDoYo11}, Hobson \cite{Ho16}, Jourdain-Zhou \cite{JoZh16}. 
	In contrast it is known from Gyöngy \cite{Gy86}, Dupire \cite{Du94}, and ultimately Lowther \cite{Lo08b} that Brownian motion is the unique continuous strong Markov martingale with Brownian marginals. 

	We took this as a challenge to construct  examples of a ``very fake'' Brownian motion, that is, \emph{continuous Markov martingales with Brownian marginals} that miss out only on the strong Markov property. %We also present an elegant variant of this construction which was kindly shown to us by George Lowther.
	
\smallskip

\noindent \emph{keywords:} 
fake Brownian motion, mimicking processes, Markov property
\end{abstract}

\section{Overview}\label{Introduction}
In this paper, we show that there exist continuous Markov martingales which have the same marginals as Brownian motion but are different from Brownian motion: 

%The aim of this article is to construct a fake Brownian motion $X$:

\begin{theorem}
	\label{continuous}
	There is a 1-dimensional Markovian martingale $X$ with continuous paths and Brownian marginals which is not strongly Markovian.
\end{theorem}

The main part of the paper is devoted to a construction of such a fake Brownian motion $X$ that relies entirely on familiar  techniques of stochastic analysis. While the basic idea is quickly explained (see Section \ref{BareHands}), it will then require some work to fill in the details of this construction. 

% GL: updated this paragraph
%We provide two distinct methods. The first of these is a direct construction of the process, although some effort is required to show that it satisfies the claimed properties.

Alternatively, we provide second construction in the final Section \ref{GLSection} which is less explicit, but is handled quickly and very easily by a theorem established in \cite{Lo08b}.

%Readers who find this approach too laborious are invited to jump directly to the final section  \ref{GLSection}, where we give an alternative, more elegant construction of a continuous Markovian fake Brownian motion based on a deep theorem established in  \cite{Lo08b}.

\section{Bare hands approach to  faking Brownian motion}\label{BareHands}
Fix a 1-dimensional Brownian motion $B = (B_t)_{0\le t < \infty}$. In this  section we shall assume that the starting value $B_0$ is normally distributed with mean $0$ and variance $1$, in symbols $N(0,1)$, so that $B_t$ is $N(0,1+t)$-distributed. We denote by $p_t(\cdot)$ its density function.

% \begin{center}{\bf GL: Suggested changes}\end{center}

Let $C$ be a Borel subset of the interval $[-1,1]$ and set $G=\mathbb R\setminus C$.
The accumulated amount of time the original Brownian motion $(B_s)_{0 \leq s \leq t}$ spends inside $G$ is given by
\begin{equation}\label{A}
A_t := \int_0^t \mathds 1_{\{B_s \in G\}} \, ds, \qquad 0 \le t < \infty,
\end{equation}
which defines an increasing, Lipschitz-1-continuous process, c.f.\ \cite[Chapter III.21]{RoWi00}.
Its right continuous inverse 
\begin{equation}\label {tau}
\tau_t := \inf \{ s > 0 \colon A_s > t\}, \qquad 0 \le t < \infty,
\end{equation}
defines a time-change of the Brownian motion $B$.
The resulting process $(B_{\tau_t})_{0 \leq t < \infty}$ 
is a strongly Markovian martingale taking values in $G$ almost everywhere.
Consider the case where $G$ is a union of finitely many intervals.
Intuitively speaking, as long as $B_{\tau_t}$ takes values in the interior of one of the above intervals, this process behaves like a Brownian motion.
When it hits the boundary of the interval it is either reflected back into its interior or it jumps to the corresponding boundary of the neighbouring interval.

To imitate Brownian marginals not only on $G$, but also on its compliment $C$, we write $U$ for a uniformly distributed random variable on $[0,1]$ independent of $B$.
With the aid of $U$ we introduce a stopping time $T$
\begin{equation}
	\label{T}
	T := \begin{cases} 0 & B_0 \in G, \\ \inf \left\{ t > 0 \colon U \geq \frac{p_t(B_0)}{p_0(B_0)}  \right\} & B_0 \in C,\end{cases}
\end{equation}
which will allow us in the following to differentiate at each time $0 \leq t < \infty$ between two different species of particles as will be explained in a moment.
Speaking formally, the process of interest $X$ can be defined by virtue of $T$, and is given by
\begin{equation}
	\label{X}
	X_t := \begin{cases} B_{\tau_{t - T}} & t \geq T, \\ B_0 & \text{else.} \end{cases}
\end{equation}
We consider $X$ with respect to its natural (right-continuous, saturated) filtration $(\mathcal F_t)_{0 \leq t < \infty}$.
\begin{proposition} \label{GLprop}
	The process $X$ has the same one-dimensional marginals as the Brownian motion $B$.
	It is a Markov martingale with c\`adl\`ag paths and, if $C$ is closed with empty interior, is continuous.
	\end{proposition}

The formal proof will be given in Section \ref{sec:proofs} below.

We note that $T > 0$ almost surely on $C$, and that the process $X$ is constant on the interval $[0,T]$.
Hence, $X$ is certainly not a Brownian motion so long as $C$ has positive Lebesgue measure.
As an example, if $C$ is a fat Cantor set (see e.g.\ \cite[page 140]{AlBu98}) then it has positive measure yet has empty interior, so that $X$ is a continuous Markov fake Brownian motion.

% {\bf GL: comments}

% \begin{itemize}
% \item the description above could go in section 2, but could also be moved to section 3, and restore the original text in section 2.
% \item the case where $G$ is a union of finitely many intervals is exactly the same as in the original text, so the proof should follow in the same way.
% \item the case where $G$ is open should also follow in the same way as for intervals. The process behaves as a Brownian motion in $G$.
% \item for an arbitrary Borel $G$, we can choose a decreasing sequence of open sets $G^N$ such that $\bigcap_NG^N\setminus G$ has zero measure, and the proposition follows by taking limits.
% \item we do not really need $C$ to have empty interior. It is sufficient for $G\cap(a,b)$ to have nonzero Lesbegue measure for all nontrivial intervals $(a,b)$. Continuity follows so long as $A$ is never constant over any nontrivial interval.
% \end{itemize}

% \begin{center}{\bf GL: End of suggested changes}\end{center}

We know since Bachelier's thesis (compare the twin paper \cite{BePaSc21b} which also contains much of the motivation for the present construction, in particular Section 3) that, at time $t \ge 0$, the {\it net inflow} of $B$ into the interval $[a,\infty[$ at the left end point $a$ equals $-\frac{p'(t,a)}{2}$, i.e.,

\begin{equation}\label{Bachelier}
\frac{\partial}{\partial t} \mathbb P [B_t \ge a] = -\frac{p'(t,a)}{2},
\end{equation}
as follows immediately from the heat equation by integrating with respect to the space variable.

If an interval $[a,b]$ is contained in $]0,\infty[$ we therefore find a positive net inflow at the left boundary $a$ and a negative inflow, i.e. a net outflow, at the right boundary $b$.

In order to construct a fake Brownian motion $X = (X_t)_{0 \le t < \infty}$, let us focus our attention on an interval $[a,b] \subseteq \, ]0,\infty[$ for the moment.  The idea of the present construction is that the process $X$ behaves like a Brownian motion as long as $X_t$ takes its values in the interior of $[a,b]$. When $X$ hits the boundaries we have to make sure that the the amount of net in- resp.~out-flows into the interval $[a,b]$ agree with the values for the original Brownian motion as given by (\ref{Bachelier}). If we can do so, the evolution of the marginal densities of the Brownian motion $B$ and the fake Brownian motion $X$ will coincide on $[a,b]$. We still note that the process $X$ may jump into, resp.~out of, the interval $[a,b]$ at its boundary points, or it may move in or out in a continuous way.

We shall do our construction in two steps. First we consider finitely many disjoint intervals $[a_1^N,b_1^N], \dots,[a_N^N,b_N^N]$ in $]0,1[$ and achieve on each interval the validity of the above program to arrive at a process $X^N$ with the proper marginals on these intervals. This process will jump between the boundary points of neighboring intervals. We shall also have to make sure that on the remaining complement of these intervals the marginals of $X^N_t$ also coincide with the marginals $p_t$ of $B_t$.

In a second step we let $N$ go to infinity and pass to an infinite collection of disjoint intervals $[a_n,b_n]$, contained in $]0,1[$, whose union is dense in $]0,1[$, but has Lebesgue measure strictly less than one. We thus will pass to a limit $X$ of the above processes $X^N$ which will have {\it continuous trajectories} as well as the proper marginals. While the processes $X^N$ will be strongly Markovian martingales,  the limiting process $X$ will still be a Markovian martingale, but fail to have the strong Markov property.

\section{Construction of the Processes \texorpdfstring{$X^N$}{TEXT}} \label{X^N}
Let $[a_n^N,b_n^N ]_{n=1}^N$ be disjoint intervals in $]0,1[$, ordered from left to right.
We also write $]a_0^N, b_0^N] = ]-\infty,0]$ and $[a_{N+1}^N,b_{N+1}^N[ = [1,\infty[$.
Let
\[
	G^N := \bigcup_{n = 0}^{N+1} [a_n^N,b_n^N] \text{ and } C^N := \bigcup_{n = 1}^N ]b_n^N,a_{n+1}^N[,
\]
thus $C^N$ is precisely the complement of $G^N$.

The accumulated amount of time the original Brownian motion $(B_s)_{0 \leq s \leq t}$ spends inside $G^N$ is given by
\begin{equation}\label{A^N}
A^{N}_t := \int_0^t \mathds 1_{\{B_s \in G^N\}} \, ds, \qquad 0 \le t < \infty,
\end{equation}
which defines an increasing, Lipschitz-1-continuous process, c.f.\ \cite[Chapter III.21]{RoWi00}.
% which is the amount of time the Brownian motion $(B_s)_{0 \le s \le t}$ spends in the set $G^N$. 
% (compare [Rogers Williams ....])
Its right continuous inverse 
\begin{equation}\label {tau^N}
\tau^{N}_t := \inf \{ s > 0 \colon A^{N}_s > t\}, \qquad 0 \le t < \infty,
\end{equation}
defines a time-change of the Brownian motion $B$. The resulting process $(B_{\tau_t^N})_{0 \leq t < \infty}$ 
is a strongly Markovian martingale taking values in $G^N$.
Intuitively speaking, as long as $B_{\tau_t^N}$ takes values in the interior of one of the above intervals, this process behaves like a Brownian motion.
When it hits the boundary of the interval it is either reflected back into its interior or it jumps to the corresponding boundary of the neighbouring interval.
% As is well known, the intensity rate of these jumps is proportional to the local time spent by the particle $B_{\tau^N_t(\omega)}(\omega)$ at the respective boundary points, and indirectly proportional to the distance from the neighbouring interval, i.e.\ to the size of this jump.

To imitate Brownian marginals not only on $G^N$, but also on its compliment $C^N$, we write $U$ for a uniformly distributed random variable on $[0,1]$ independent of $B$. 
With the aid of $U$ we introduce a stopping time $T^N$
\begin{equation}
	\label{TN}
	T^N := \begin{cases} 0 & B_0 \in G^N, \\ \inf \left\{ t > 0 \colon U \geq \frac{p(B_0,t)}{p(B_0,0)}  \right\} & B_0 \in C^N,\end{cases}
\end{equation}
which will allow us in the following to differentiate at each time $0 \leq t < \infty$ between two different species of particles as will be explained in a moment.
Speaking formally, the process of interest $X^N$ can be defined by virtue of $T^N$, and is given by
\begin{equation}
	\label{XN}
	X^N_t := \begin{cases} B_{\tau^N_{t - T^N}} & t \geq T^N, \\ B_0 & \text{else.} \end{cases}
\end{equation}
We consider $X^N$ with respect to its natural (right-continuous, saturated) filtration $(\mathcal F_t)_{0 \leq t < \infty}$.
\begin{proposition} \label{finite}
	The process $X^N$ has the same one-dimensional marginals as the Brownian motion $B$.
	It is a strong Markov martingale with c\`adl\`ag paths, but fails to be continuous.
	\end{proposition}

The formal proof will be given in Section \ref{sec:proofs} below. Here we only sketch the main ideas.
The verification of the strong Markovianity of $X^N$ is rather straight-forward.
The crucial issue pertains to the marginals of $X^N$.

To verify that the marginals of $X^N$ are indeed Brownian, we distinguish between the behavior of the process $X^N$, whether it takes values in $G^N$ or in $C^N$.
We call particles with $t \geq T^N$ {\it busy particles} as opposed to the {\it lazy particles}, which remain at their initial position in $C^N$ until time $T^N$ and to which we now turn our attention.

The idea is -- as indicated by the word ``lazy'' -- that these particles do not move for some time.
Eventually, namely at time $T^N$, they will change their behaviour from the ``lazy'' state to follow the behaviour of the busy particles.
For $x \in C^N$ the density function $p(x,t)$ of the Brownian motion $(B_t)_{t \geq 0}$ is decreasing in time, and its infinitesimal change is given by the formula
\begin{equation}
	\label{lazy mass change}
	\frac{\partial}{\partial t} p(x,t) = \left( \frac{x^2}{2(t+1)^2} - \frac{1}{2(t+1)}\right) p(x,t) < 0.
\end{equation}
In order to match the evolution of the marginals of the fake Brownian motion $X^N$ with that of the original Brownian motion $B$ on the set $C^N$, we note that by \eqref{lazy mass change} the fraction $\frac{p(x,t)}{p(x,0)}$ is strictly decreasing in time for any fixed $x \in C^N$. Thus,
\begin{equation}
	\label{conditional mass change on CN}
	\mathbb P\left(t \leq T^N \mid B_0 = x \right) = \mathbb P\left( U < \frac{p(x,t)}{p(x,0)} \right) = \frac{p(x,t)}{p(x,0)}.
\end{equation}
As the trajectories of the busy particles (that is $t \geq T^N$) take values exclusively in $G^N$, it is apparent from \eqref{conditional mass change on CN} that the correct amount of (lazy) particles stay at their starting position to match the Brownian marginals on $C^N$.

To further develop our intuition, consider the Brownian motion $B$ conditionally on the starting value $B_0 = x \in \, ]b_{n-1}^N,a_{n}^N[ \subseteq C^N$ for $1 \leq n \leq N + 1$.
In other words, $[a_{n-1}^N,b_{n-1}^N]$ and $[a_n^N,b_n^N]$ are the neighbouring intervals which lie to the
left, resp.\ to the right of the starting point $x$.
As $x \in C^N$, the Brownian motion $B$ with starting value $x$ does not spend any time in $G^N$ before it hits one of
the boundary points $\{b_{n-1}^N,a_n^N\}$, and from this hitting time onwards, it will almost surely spend its time in $G^N$.
As a consequence, the time changed c\`adl\`ag process $(B_{\tau^N_t})_{t \geq 0}$ does not start at $x$ but rather at one of the boundary points $\{b_{n-1}^N,a_{n}^N\}$.
It takes its choice of these two possibilities with probability
\begin{equation}
	\label{Z^N jump probabilities}
	\mathbb P(B_{\tau_0^N} = a_n^N \mid B_0 = x) = \frac{x - b_{n-1}^N}{a_{n}^N - b_{n-1}^N},\quad \mathbb P(B_{\tau_0^N} = b_{n-1}^N \mid B_0 = x) = \frac{a_{n}^N-x}{a_{n}^N - b_{n-1}^N}.
\end{equation}
This causes the process $X^N$ to jump at time $T^N$ from $x$ to the boundary $\{b_{n-1}^N,a_n^N\}$
with the correct probabilities turning $(X^N_{T \wedge t})_{0 \leq t < \infty}$ into a martingale.
After this jump the particle enters into the ``busy'' mode and exhibits the same behaviour as the Markov martingale $(B_{\tau^N_t})_{0 \leq t < \infty}$ taking values in $G^N$.

The analysis of the marginal flow on $G^N$ is more delicate.
We shall use some formal arguemnts in the remainder of this section to reveal the intuition behind our construction.
A rigorous treatment is postponsed to Section \ref{sec:proofs}.
A ``busy'' particle $X^N_t$ takes values in either of the intervals $([a^N_n,b_n^N])_{n = 0}^{N+1}$ which compose $G^N$.
In this case $X^N$ behaves like a Brownian motion inside of $[a^N_n,b_n^N]$ as long as it moves in the interior of one of the intervals up to hitting the boundary.
As mentioned, at this stage the particle can be either reflected back into the interval, or jump into the neighboring interval.
As is well known, the intensity rate of these jumps is proportional to the local time spent by the particle $X_t^N(\omega)$ at the respective boundary points, and indirectly proportional to the distance from the neighbouring interval, i.e.\ to the size of this jump. 
Anticipating that $X^N$ has at time $t$ the correct marginal distribution, this observation leads to
\begin{align}
	\frac1{dt}\mathbb P\left( X_{t + dt} = a_n^N, X_t \in [a_{n-1}^N, b_{n - 1}^N] \right) = \frac{p(b_{n-1}^N,t)}{2(a_n^N - b_{n - 1}^N)}, \label{1}
\\	\frac1{dt}\mathbb P\left( X_{t + dt} = b_{n-1}^N, X_t \in [a_{n}^N,b_n^N] \right) = \frac{p(a_n^N,t)}{2(a_n^N - b_{n - 1}^N)}, \label{2}
\end{align}
where \eqref{1} describes an inflow at $a_n^N$ from the neighbouring interval $[a_{n-1}^N,b_{n-1}^N]$ whereas \eqref{2} describes an outflow from $[a_n^N,b_n^N]$ to  $[a_{n-1}^N,b_{n-1}^N]$.

Simultaneously the boundary points in $G^N$ experience an additional mass inflow caused by ``lazy'' particles switching to the ``busy'' regime at time $T^N$.
Thanks to \eqref{Z^N jump probabilities} we can explicitly derive the inflow at $a_{n+1}^N$ caused by particles changing their behavior from the ``lazy'' to the ``busy'' mode.
The heat equation and an integration by parts yield
\begin{align} \nonumber
	\frac1{dt}\mathbb P\left( X_{t+dt}^N = a_n^N, X_t^N \in (b_{n-1}^N,a_n^N) \right) &= -\int_{b_{n-1}^N}^{a_n^N} \frac{x - b_{n-1}^N}{a_n^N - b_{n-1}^N} \frac{\partial}{\partial t} p(x,t) \, dx
\\	&= -\frac{1}{2} \int_{b_{n-1}^N}^{a_n^N} \frac{x - b_{n-1}^N}{a_n^N - b_{n-1}^N} \frac{\partial^2}{\partial x^2} p(x,t) \, dx \nonumber
\\	&= \frac{1}{2} \left( -\frac{\partial}{\partial x} p(a_n^N,t) + \frac{p(a_n^N,t) - p(b_{n-1}^N,t)}{a_n^N - b_{n-1}^N}\right). \label{3}
\end{align}
Now we arrive at the crucial point of the construction: adding the effect of \eqref{3} to the in \eqref{1} and \eqref{2} calculated in- and out-flows of the ``busy'' particles we arrive {\it precisely} at the net inflow of the original Brownian motion $B$ at the boundary point $a_n^N$.
Of course, a similar argument applies to the right boundary point $b_n^N$ as well to all the other boundary points.
In conclusion, the marginals of $X^N_t$ equal the marginals of $B_t$ on $C^N$ as well as on $G^N$, i.e.\ on all of $\R$. 
This finishes the intuitive sketch of the ideas underlying the proof of Proposition \ref{finite} indicating that we have successfully constructed a fake Brownian motion with the properties detailed in Proposition \ref{finite}.
In Section \ref{sec:proofs} we shall translate this intuition into rigorous mathematics.

\section{Construction of the Process \texorpdfstring{$X$}{TEXT}} \label{sec:continuous}

The above sequence $(X^N)_{N=1}^\infty$ of processes will allow us to pass to a limiting {\it continuous} process $X$ which will be our desired fake Brownian motion.

As already mentioned in the introduction, we choose an infinite collection of closed disjoint intervals $([a_n,b_n])_{1 \le n < \infty}$  in $[0,1]$, whose union is dense in $[0,1]$ but has Lebesgue measure strictly less than one.
Taking the first $N$ intervals, ordering them from left to right, and adding the intervals $]a_0^N, b_0^N]=\, ]-\infty,0]$ and $[a_{N+1}^N,b_{N+1}^N[\, = [1,\infty[$, we are in the situation of the previous section to obtain a process $X^N$.
We denote by $G = \bigcup_{N=1}^\infty G^N$ the union of all these intervals, and the complement of $G$ by $C = \bigcap_{N = 1}^\infty C^N$.
Recall the definitions \eqref{A^N} and \eqref{tau^N}
\begin{equation}\label {tau^N,2}
A^{N}_t := \int_0^t \mathds 1_{G^N}(B_s(\omega)) \, ds,  \qquad \tau^{N}_t := \inf \{ s > 0 \colon A^{N}_s > t\}.
\end{equation}
Clearly the trajectories $(A^N_t(\omega))_{0 \le t < \infty}$ of the process $A^N$ are increasing, Lipschitz-1 continuous, and increase almost surely pointwise to the process $A$ given by
\begin{equation}\label {tau^N,3}
A_t := \int_0^t \mathds 1_{G}(B_s) \, ds, \qquad \tau_t := \inf\left\{ s > 0 \colon A_s > t \right\}.
\end{equation}
As $G$ is a union of intervals which is dense in $\R$, the original Brownian motion $B$ spends almost surely a positive amount of time during any time window of the form $[t_1,t_2]$, $0 \leq t_1 < t_2 < \infty$, in $G$.
For this reason $A$ is almost surely strictly increasing, and $(\tau_t)_{0 \leq t < \infty}$ has almost surely continuous paths.
The trajectories $(\tau^N_t(\omega))_{0 \le t < \infty}$ of the process $\tau^N$ decrease almost surely to the continuous trajectories $\left(\tau_t(\omega)\right)_{0 \le t < \infty}$, uniformly on compact subsets of $[0,\infty[$.

The stopping times $(T^N)_{N \in \N}$ previously defined in \eqref{TN} converge pointwise to $T$, which is given by
\begin{equation}
	\label{T}
	T := \begin{cases}0 & B_0 \in G, \\ \inf\left\{ t > 0 \colon U \geq \frac{p(B_0,t)}{p(B_0,0)}\right\} & B_0 \in C. \end{cases}
\end{equation}
In conclusion, the limiting process $X$
\begin{equation} \label{X}
	X_t := B_{(\tau_t - T) \wedge 0} = \begin{cases} B_{\tau_{t - T}} & t \geq T, \\ B_0 & \text{else,}\end{cases}
\end{equation}
has almost surely continuous trajectories and is almost surely the limit of the sequence $(X^N)_{N \in \N}$.
We shall verify that it is the desired fake Brownian motion described by Theorem \ref{continuous}. 

The proof of Theorem \ref{continuous} is given in Section \ref{sec:continuous} below.
From an intuitive point of view Theorem \ref{continuous} is a rather obvious consequence of Proposition \ref{finite}.
The novel aspect is the failure of the strong Markov property.
This failure follows from a general theorem of Lowther \cite[Theorem 1.3]{Lo08b}: an $\R$-valued, continuous, and strongly Markovian martingale is uniquely determined by its one-dimensional marginals.
In particular, if these marginals are those of a Brownian motion $B$, there is no other process $X$ with the mentioned properties.
In other words, there is no fake Brownian motion which is a continuous, {\it strong} Markov martingale.

\begin{corollary}
	The fake Brownian motion $X$ fails to have the strong Markov property.
\end{corollary}

It is instructive to directly visualize the failure of the strong Markov property of the above constructed process $X$, without recurse to Lowther's theorem. We do so by applying the technique of {\it coupling} as in \cite{Ho98c}. 

Let $\tilde X$ be an independent copy of the process $X$ given by \eqref{X}, and let $\tilde T$ be the corresponding independent copy of $T$, see \eqref{T}, which indicates the time when $\tilde X$ changes from ``lazy'' mode to the ``busy'' mode.
Both are defined on the same probability space which we equip with the filtration $({\mathcal{F}_t})_{0 \le t < \infty}$, generated by these two processes, whose elements we write as $(\omega, \tilde \omega)$.
Define the stopping time $\tau = \tau(\omega, \tilde \omega)$ as the first moment when the trajectories $X(\omega)$ and $\tilde X(\tilde \omega)$ meet, that is
\begin{equation}
	\label{tau}
	\tau(\omega,\tilde \omega) := \inf\left\{ t > 0 \colon X_t(\omega) = \tilde X_t(\tilde \omega) \right\}.
\end{equation}

On the event $\{T(\omega) \leq \tau(\omega,\tilde \omega)\}$ we know that $X$ is in the ``busy'' regime from time $\tau$ onwards.
That means that $(X_{\tau + t})_{0 \leq t < \infty}$ behaves like the strong Markov process $(B_{\tau_t})_{0 \leq t < \infty}$ when the underlying Brownian motion has the correct starting distribution $B_0 \sim X_{\tau}$.
Hence we have, almost surely for any $t \in [0,\infty[$ on $\{T(\omega) \leq \tau(\omega,\tilde \omega)\}$, that
\begin{equation} \label{strong Markov when busy}
	\mathbb E\left[ \mathds 1_{G}(X_\tau + t) \mid \mathcal F_{\tau} \right] = \mathbb E\left[ \mathds 1_{G}(B_{\tau_t}) \mid B_0 = X_\tau \right].
\end{equation}
The sets $\mathcal A, \mathcal B \in \mathcal{F}_\tau$ where the particle $X_\tau(\omega)$ (resp.\ $\tilde X_\tau(\tilde \omega)$) at time $\tau$ is in the ``busy'' mode while  ${\tilde X}_\tau(\tilde \omega)$ (resp.\ $X_\tau(\omega)$) is in the ``lazy'' mode, in symbols
\begin{align}
	\label{calA}
	\mathcal A :=& \left\{(\omega,\tilde \omega) \colon T(\omega) \leq \tau(\omega,\tilde \omega) < \tilde T(\tilde \omega) \right\},
\\	\label{calB}
	\mathcal B:=& \left\{ (\omega,\tilde \omega) \colon \tilde T(\tilde \omega) \leq \tau(\omega,\tilde \omega) < T(\omega) \right\}.
\end{align}
have positive probability and $\mathbb P\left( \mathcal A \right) = \mathbb P \left( \mathcal B \right)$.
Define the conditional probabilities $\mathbb P_{\mathcal A}(\cdot \cap \mathcal A)$ and $\mathbb P_{\mathcal B}(\cdot \cap \mathcal B)$, and particularly due to symmetry
\begin{equation}
	\label{distribution at time tau}
	\mu(dx,dt) := \mathbb P_\mathcal A\left( X_\tau \in dx, \tau \in dt\right) = \mathbb P_\mathcal B\left( X_\tau \in dx, \tau \in dt\right).
\end{equation}

The conditional probability of $X_{\tau + t}$ with respect to $\mathcal{F}_\tau$ does not only depend on the present position ${X}_\tau(\omega)) = {\tilde X}_\tau(\tilde \omega) $ but also on the information whether ${X}_\tau(\omega) $ is in the ``busy'' or ``lazy'' mode.
Indeed, fix $t > 0$ and consider the probability of the event $\{ X_{\tau + t} \in G \}$ conditionally on $\mathcal A$, which is, by \eqref{strong Markov when busy} and as $\tau^N_t \searrow \tau_t$ almost surely, given by
\begin{align*}
	\mathbb P_\mathcal A\left( X_{\tau + t} \in G \right) &= \mathbb E_{\mathbb P_\mathcal A} \left[ \mathbb P\left( B_{\tau_t} \in G \mid \mathcal F_\tau \right) \right] = \mathbb E_{\mathbb P_\mathcal A} \left[ \mathbb E \left[ \mathds 1_G (B_{\tau_t}) \mid B_0 = X_\tau \right] \right]
\\	&\geq \mathbb E_{\mathbb P_\mathcal A} \left[ \mathbb E \left[ \mathds 1_{\text{int}(G)}(B_{\tau_t}) \mid B_0 = X_\tau \right] \right] 
\\ &= \lim_{N \to \infty} \mathbb E_{\mathbb P_\mathcal A}\left[ \mathbb P\left( B_{\tau_t^N} \in \text{int}(G) \mid B_{0} = X_\tau \right) \right] = 1,
\end{align*}
where we write $\text{int}(G)$ for the interior of $G$.
We deduce that
\begin{equation}
	\label{cond prob1}
	\mathbb P_\mathcal A\left( X_{\tau + t} \in C \mid X_\tau, \tau\right) = 0, \quad \text{thus,} \quad \mathds 1_{\mathcal A} \mathbb P\left( X_{\tau + t} \in C \mid \mathcal F_\tau \right) = 0,
\end{equation}
whereas,
\begin{equation}
	\label{cond prob2}
	\mathbb P_\mathcal B\left( X_{\tau + t} \in C \mid X_\tau, \tau \right)>0,
\end{equation}
thus $\mathds 1_{\mathcal B}\mathbb P\left( X_{\tau + t} \in C \mid \mathcal F_\tau\right)$ does not vanish.
From \eqref{distribution at time tau}, \eqref{cond prob1}, and \eqref{cond prob2}, we follow that in order to determine the conditional probability of $\{X_{\tau + t} \in C\}$ given $\mathcal F_\tau$ we require the information from the past of the process $X$ prior to the stopping time $\tau$.
In conclusion, the process $X$ fails to have the strong Markov property.

Intuitively speaking, this failure stems from the fact that the ``busy'' particles travel through the ``lazy territory'' $C$ in a continuous way. In contrast, the processes $X^N$ considered in the previous section jump over the territory $C^N$ which makes it impossible to ``catch'' them while they are traveling through $C^N$  by a stopping time $\tau$.

\section{Proofs} \label{sec:proofs}

\begin{lemma} \label{BtauN is Feller}
	Using the notation of Section \ref{X^N}, let $B$ be a Brownian motion started at some $B_0 = x \in G^N$.
	Then the time-changed Brownian motion $(B_{\tau^N_t})_{0 \leq t < \infty}$ is a Feller process.
	Its Feller generator $\mathcal G$ is given by
	\begin{equation} \label{generator}
		\mathcal Gf(x) = \frac{1}{2} \frac{\partial^2}{\partial x^2} f(x)
	\end{equation}
	and its domain by
	\begin{equation} \label{domain}
		\left\{f \in C^2(G^N) \colon \partial_- f(b_{n-1}^N) = \frac{f(a_n^N) - f(b_{n-1}^N)}{a_n^N - b_{n-1}^N}= \partial_+ f(a_n^N), \text{ for every }n = 1,\ldots,N+1\right\}.
	\end{equation}
\end{lemma}

\begin{proof}
	To see that $(B_{\tau_t^N})_{0 \leq t < \infty}$ defines a Feller process, we recall that it inherits the strong Markov property from $B$.
	Define its resolvent
	\[
		R_\lambda f(x) := \int_0^\infty e^{-\lambda t} \mathbb E\left[f(B_{\tau^N_t}) \mid B_0 = x\right] \,dt.
	\]
	By the Hille-Yosida theorem, it suffices to show that $(R_\lambda)_{0 < \lambda < \infty}$
	is a strongly continuous contraction resolvent on $C_0(G^N)$ (the set of continuous functions on $G^N$ vanishing at $\pm \infty$).
	Recall the definition of the process $(A_t^N)_{0 \leq t < \infty}$, see \eqref{A^N}.
	Conditionally on the starting point $B_0 = x \in G^N$, we have as a consequence of Blumenthal's 0-1 law that almost surely $\forall t > 0$, $A_t^N > A_0^N$.
	Therefore, as $t \mapsto A_t^N$ is increasing and continuous (particularly at $0$), we find
	\[
		\lim_{t \searrow 0}\tau^{N}_t = \lim_{t \searrow 0} A^N_t = 0.
	\]
	Thus, for any $x \in G^N$
	\[
		\lim_{t \searrow 0}\mathbb E\left[f(B_{\tau^N_t}) \mid B_0 = x\right] = f(x) \text{ and }\lim_{\lambda \to \infty} \lambda R_\lambda f(x) = f(x).
	\]
	Due to \cite[Lemma 6.7 and its proof]{RoWi00}, it remains to show that $R_\lambda$ maps $C_0(G^N)$ into itself.
	Let $x,y \in [a_n^N,b_n^N] \subseteq G^N$, then we obtain as in \cite[Equation (22.2)]{RoWi00}
	\begin{equation}
		\label{continuity of resolvent}
		|R_\lambda f(x) - R_\lambda f(y)| \leq \frac{2\lVert f\rVert}{\lambda} \mathbb E\left[1 - e^{-\lambda \tilde H_y} \mid B_0 = x \right],
	\end{equation}		
	where $\tilde H_y := \inf\{ t > 0 \colon B_{\tau_t^N} = y\}$ is the first hitting time of $y$ and $\lVert \cdot \rVert$ denotes the supremum norm on $C_0(G^N)$.
	Since $\tilde H_y$ is dominated by $H_y := \inf\{t > 0 \colon B_t = y\}$, we have that for $\epsilon > 0$
	\[
		\lim_{y \to x}\mathbb P\left( \tilde H_y \geq \epsilon \mid B_0 = x \right) = 0,
	\]
	whence, the right-hand side in \eqref{continuity of resolvent} vanishes.
	We have shown that the time-changed process $(B_{\tau_t^N})_{ 0 \leq t < \infty}$ is a Feller process.

	Next, we compute its generator. 
	Let $x \in \, ]a_n^N,b_n^N[$ and recall that the expected time the Brownian motion $B$ started at $B_0 = x$ spends inside $[x-\epsilon, x + \epsilon] \subseteq [a_n^N,b_n^N]$ before exiting $[x-\epsilon, x + \epsilon]$ is given by 
	\[ \mathbb E\left[ H_{x - \epsilon} \wedge H_{x + \epsilon} \mid B_0 = x \right] = \epsilon^2. \]
	Due to Dynkin's formula we can compute the generator via
	\begin{align*}
		\mathcal Gf(x) &= \lim_{\epsilon \searrow 0} \frac{\frac{f(x + \epsilon) + f(x - \epsilon)}{2} - f(x)}{\mathbb E\left[H_{x -\epsilon} \wedge H_{x + \epsilon} \mid B_0 = x \right]}
	\\	&= \lim_{\epsilon \searrow 0}\frac{f(x + \epsilon) - f(x) + f(x - \epsilon) - f(x)}{2\epsilon^2}
	\\	&= \lim_{\epsilon \searrow 0}\frac{f'(x + \epsilon) - f'(x - \epsilon)}{4\epsilon} = \frac{f''(x)}{2},
	\end{align*}
	and note that $\mathcal Gf(x)$ can only exist if $f$ is two times continuously differentiable at $x$.

	Next let $x = a_n^N$, $l = a_n^N - b_{n-1}^N$ for some $n = 1,\ldots, N+1$.
	The expected time a Brownian motion $B$ started at $B_0 = 0$ spends inside $[0,\epsilon]$ before $H_\epsilon$ is $\epsilon^2$.
	Therefore, we can compute the expected time the Brownian motion $B$ started at $B_0 = x$ spends inside $[x,x+\epsilon] \subseteq [x = a_n^N,b_n^N]$ before hitting either of $\{b_{n-1}^N,x + \epsilon\}$ for the first time:
	\begin{align*}
		\epsilon^2 &= \mathbb E\left[ \int_0^{H_{x + \epsilon}} \mathds 1_{[x,x+\epsilon]}(B_s) \, ds \mid B_0 = x \right]
		\\	&= \mathbb E\left[\int_0^{H_{x + \epsilon} \wedge H_{b_{n-1}^N}} \mathds 1_{[x,x+\epsilon]}(B_s) \, ds \mid B_0 = x \right]	+ \mathbb P\left(H_{b_{n-1}^N} < H_{x + \epsilon} \mid B_0 = x\right) \epsilon^2
		\\	&= \mathbb E\left[ \tilde H_{b_{n-1}^N} \wedge \tilde H_{x + \epsilon}\right] + \frac{\epsilon^3}{l + \epsilon}.
	\end{align*}
	Again, Dynkin's formula allows us to compute the generator at $x$
	\begin{align*}
		\mathcal Gf(x) &= \lim_{\epsilon \searrow 0} \frac{\mathbb E\left[ f(B_{\tau^N(\tilde H_{b_{n - 1}^N} \wedge \tilde H_{x + \epsilon})}) \mid B_0 = x\right] - f(x)}{\mathbb E\left[ \tilde H_{b_{n - 1}^N} \wedge \tilde H_{x + \epsilon} \mid B_0 = 0 \right]}
	\\	&= \lim_{\epsilon \searrow 0} \frac{ \frac{l}{l + \epsilon} \left( f(x + \epsilon) - f(x) \right) +\frac{\epsilon}{l + \epsilon} (f(b_{n-1}^N) - f(x))}{ \epsilon^2 ( 1 - \frac{\epsilon}{b + \epsilon})}.
	\end{align*}
	This limit can only exist if $f'(x)$, $f''(x)$ exist (where $f'(x)$ and $f''(x)$ denote here the adequate one-sided derivatives) and $lf'(x) = f(x) - f(b_{n-1}^N)$.
	In this case, we have by de l'Hopital's rule
	\begin{align*}
		\mathcal Gf(x) &= \lim_{\epsilon \searrow 0} \frac{ \frac{l}{l+\epsilon} f'(x+\epsilon) + \frac{f(b_{n-1}^N) - f(x)}{l + \epsilon} - \frac{1}{(l+\epsilon)^2}\left( \epsilon (f(b_{n-1}^N  - f(x)) + l(f(x+\epsilon) - f(x)) \right) }{2\epsilon}
	\\	&= \frac{ f''(x) - \frac{f'(x)}{l} - \frac{2}{l^2}(f(b_{n-1}^N) - f(x)) - \frac{f'(x)}{l}  }{2} = \frac{f''(x)}{2}.
	\end{align*}
	We conclude by remarking that by analogous reasoning we have for $b_n^N$ and some $n = 0,\ldots,N$ that $f$
	can only be in the domain of $\mathcal G$ if and only if it is two times (one-sided) differentiable at $b_n^N$ and
	\begin{align*}
		\mathcal Gf(b_n^N) = \frac{f''(b_n^N)}{2},\quad f'(b_n^N) = \frac{f(a_{n+1}^N) - f(b_n^N))}{a_{n+1}^N - b_n^N},
	\end{align*}
	where $f'(b_n^N)$ and $f''(b_n^N)$ are again the adequate one-sided derivatives.
\end{proof}

% By the same reasoning as in Lemma \ref{BtauN is Feller} with appropriate adaptions, we derive that $(B_{\tau_t})_{0 \leq t < \infty}$ is a Feller process.
% We remark that the Feller process $(B_{\tau_t})_{0 \leq t < \infty}$ can be started at any $B_{\tau_0} = x \in \R$ whereas $(B_{\tau^N_t})_{0 \leq t < \infty}$ can only be started in $B_{\tau^N_0} = x \in G^N$.

\subsection{Marginals} \label{sec:marginals}

This section is concerned with the verification that $X^N$ has the correct Brownian 1-dimensional marginals.
On a formal level one may argue as follows:
We have shown in Lemma \ref{BtauN is Feller} that the busy particles $(B_{\tau^N_t})_{0 \leq t < \infty}$ behave like a Feller process with generator $\mathcal G$, c.f.\ \eqref{generator} and \eqref{domain}.
As we know the exact form of the generator, it is possible to derive the Kolmogorov forward equation describing the time evolution of the density $u$ of the process $(B_{\tau^N_t})_{0 \leq t < \infty}$, that is,
\begin{gather*}
	\frac{\partial}{\partial t} u(x,t) = \frac{\partial^2}{\partial x^2} u(x,t) \quad (x,t) \in (G^N \setminus \partial G^N)\times (0,\infty[,
\\	\frac{\partial}{\partial x} u(a_n^N,t) = \frac{u(a_n^N,t) - u(b_{n-1}^N)}{a_n^N - b_{n-1}^N} = \frac{\partial}{\partial x} u(b_{n-1}^N,t)\quad t \in (0,\infty[,
\end{gather*}
when $n = 1,\ldots,N$.
Thus, relying on knowledge of the corresponding heat equations one can show that the inflow of particles from $C^N$ caused by changing their modes from ``lazy'' to ``busy'' yield at the boundary $\partial G^N$ the correct compensation, whence, the density $v$ of $X^N$ satisfies the heat equation
\[
	\frac{\partial}{\partial t} v(x,t) = \frac{1}{2} v(x,t)\quad (x,t) \in \R \times (0,\infty[,	
\]
with initial condition $v(x,0) = p(x,0)$.

Laying regularity questions aside, the above sketched approach seems rather clear-cut.
Nevertheless, to avoid subtle arguments justifying the formal reasoning we ``go back to the roots'' and use a discretization argument instead.
Approximating $B$ by a scaled random walk $B^m$ allows us to establish the form of the marginals of $X^N$ without having to worry about regularity of the involved densities.

To this end, consider a random walk on $\Z$ with i.i.d.\ increments $(\zeta_k)_{k \in \N}$ where
\begin{equation}
	\label{discrete increments}
	\mathbb P\left( \zeta_k = \pm 1 \right) = \frac{1}{2} \text{ and }\mathbb P \left( \zeta_k = 0 \right) = \frac{1}{2},\quad k \in \N,
\end{equation}
so that for any $j \in \Z$ with $-l \leq j \leq l \in \N$
\begin{equation}
	\label{discrete normal distribution}
	\tilde p_l(j) := \mathbb P\left( \sum_{k = 1}^l \zeta_k = j \right) = 2^{-2l} \binom{2l}{l+j},
\end{equation}
which satisfies the discrete heat equation
\begin{equation}
	\label{discrete heat equation}
	\tilde p_{l+1}(j) - \tilde p_l (j) = \frac{1}{4} \tilde p_l (j-1) + \frac{1}{4} \tilde p_l (j+1) - \frac{1}{2} \tilde p_l (j).
\end{equation}

From Donsker's theorem we know that the scaled random walk
\begin{equation}
	\label{donsker}
	B_t^m := \sqrt{\frac{2}{m}} \sum_{k = 1}^{\lfloor m(1 + t) \rfloor} \zeta_k, \quad t\in [0,\infty[,
\end{equation}
converges in law to the original Brownian motion $B$ as random variables on the Skorokhod space $\mathcal D([0,\infty[)$.
Recall that $B_0$ is normally distributed with mean $0$ and variance $1$, for which reason the sum in \eqref{donsker} runs from $1$ to ${\lfloor m(1 + t) \rfloor}$ rather than ${\lfloor mt \rfloor}$.
Consequentially we observe for the 1-dimensional marginal distributions that $B_t^m$ converges in law to $B_t$, and write $p^m(t,x) := \mathbb P\left( B_t^m = x \right)$ for the discrete mass evolution of $B^m$.

Let $U$ be a uniformly distributed random variable independent of $(\zeta_k)_{k \in \N}$, and define in analogy to $T^N$, c.f.\ \eqref{TN}, the time $T^{N,m}$ which indicates when a particle $X^{N,m}$ changes behaviour from ``lazy'' to ``busy'':
\begin{equation} \label{TNm}
	T^{N,m} := \begin{cases} 0 & B_0^m \in G^N, \\ \inf \left\{ t > 0 \colon U \geq \frac{\tilde p^m(t, B_0^m)}{\tilde p^m(0,B_0^m)} \right\} & B_0^m \in C^N, \end{cases}
\end{equation}
The time-change $\tau^{N,m}$ is given by
\begin{equation}
	\label{tauNm}
	\tau^{N,m}_t := \inf\left\{ s > 0 \colon A^{N,m}_s > t \right\},\quad A^{n,m}_t := \int_0^t \mathds 1_{G^N}(B^{N,m}_s) \, ds.
\end{equation}
We define the time-changed scaled random walk $X^{N,m}$
\begin{equation} \label{XNm}
	X_t^{N,m} := \begin{cases} B_0^m & t < T^N, \\ B_{\tau^{N,m}_{t - T^{N,m}}}^m & t \geq T^{N,m}. \end{cases}
\end{equation}
It is evident that this defines a strong Markov process taking values in $\sqrt{\frac{2}{m}} \Z$.
The aim of the remainder of this subsection is to first establish that $X^{N,m}$ preserves the marginals of $B^m$, see Theorem \ref{XNm and Bm same 1dim dist}.
Second, we give a proof for convergence of $(X^{N,m})_{m \in \N}$ to $X^N$ as random variables on the Skorokhod space $\mathcal D([0,\infty[)$ in Theorem \ref{XNm to XN}.

\subsubsection{The marginals of \texorpdfstring{$X^{N,m}$}{TEXT}}

An advantage of the discrete level is that evolving the marginal distribution in time is achieved iteratively.
By rescaling $B^m$ and $X^{N,m}$ by a factor of $\sqrt{\frac{m}{2}}$, we obtain random walks
\begin{equation}
	\label{auxiliary random walks}
	\tilde B^m_l := \sqrt{\frac{m}{2}} B_{\frac{l}{m}}^m = \sum_{k = 1}^{m+l} \zeta_k,\quad \tilde X^{N,m}_l := \sqrt{\frac{m}{2}} X^{N,m}_{\frac{l}{m}},\quad l \in \N,
\end{equation}
on $\Z$.
Then $X^{N,m}$ and $B^m$ have the same 1-dimensional marginal if and only if $\tilde X^{N,m}$ and $\tilde B^m$ have this as well.
For the analysis of the evolution of the marginals of $\tilde X^{N,m}$, we introduce the sets
\begin{align}
	\label{discrete sets}
	\begin{split}
	&G^{N,m} := \left\{ j \in \Z \colon j \sqrt{\frac{m}{2}} \in G^N \right\},\quad C^{N,m} := \Z \setminus G^{N,m},
	\\	&\hspace{.3cm}\partial G^{N,m} := \left\{ j \in G^{N,m} \colon \exists k \in C^{N,m} \text{ s.t. } |j-k| = 1 \right\},
	\end{split}
\end{align}
on which the transition probabilities of $\tilde X^{N,m}$ exhibit different behaviour.
In the interior of $G^{N,m}$, that means here $G^{N,m} \setminus \partial G^{N,m}$ the transition probabilities behave like the one of $\tilde B^m$ and satisfy the discrete heat equation \eqref{discrete heat equation}.
From this observation it is easy to conclude that
\begin{equation}
	\label{in interior}
	\tilde X^{N,m}_l \sim \tilde B^m \implies \forall i \in G^{N,m} \setminus \partial G^{N,m} \text{ we have } \mathbb{P}\left( \tilde X^{N,m}_{l+1} = i \right) = \mathbb P\left( \tilde B^m_{l+1} = i \right).
\end{equation}

\begin{theorem} \label{XNm and Bm same 1dim dist}
	The processes $X^{N,m}$ and $B^m$ have the same 1-dimensional marginal distributions.
\end{theorem}

\begin{proof}
	The assertion is equivalent to showing that the 1-dimensional marginal distributions of the processes defined in \eqref{auxiliary random walks} taking values in $\Z$ coincide.
	We show the statement by induction, so assume that $\tilde B_{l}^m \sim \tilde X_l^{N,m}$ for a fixed $l \in \N$.

	We treat the three cases corresponding to the three sets defined in \eqref{discrete sets} each, separately:

	\begin{enumerate}[label = \arabic*)]

		\item The case $i \in G^{N,m} \setminus \partial G^{N,m}$ was already discussed in \eqref{in interior}. We have $\mathbb P(\tilde X^{N,m}_{l+1} = i) = \mathbb P(\tilde B^m_{l+1} = i) =\tilde p_{m+l+1}(i)$.

		\item Let $i \in C^{N,m}$ and denote by $j_1,j_2 \in \partial G^{N,m}$ the boundary points with $j_1 < i < j_2$ and $]j_1,j_2[ \, \cap \Z =: C_i^{N,m} \subseteq C^{N,m}$. 
		We abbreviate $x = i \sqrt{\frac{2}{m}}$ and compute the net mass change at $i$ from time $l$ to time $l+1$ of $\tilde X^{N,m}$.
		Applying Lemma \ref{lem:discrete probabilities decreasing} ensures that $t \mapsto p^m(t,x)$ is strictly decreasing as $m + l \geq 2 i^2$.
		Thus we find
	\begin{align*}
		\mathbb P\left( \tilde X_l^{N,m} = i \right) - \mathbb P\left( \tilde X_{l+1}^{N,m} = i  \right) 
		&= \mathbb P\left( X_{\frac{l}{m}}^{N,m} = x\right) - \mathbb P\left( X_\frac{l+1}{m}^{N,m} = x \right)
	\\	&= \mathbb P\left( B^{m}_0 = x, \frac{l}{m} < T^{N,m} \leq \frac{l+1}{m} \right)
	\\	&= \mathbb P\left(B_0^m = x, p^m\left(\frac{l+1}{m},x\right)\leq Up^m\left(0, x\right)<p^m\left(\frac{n}{m}, x \right) \right)
	\\	&= \mathbb P\left( \tilde B^m_l = i\right) - \mathbb P\left( \tilde B^m_{l+1} = i\right)  
	\\	&= \tilde p_{m+l}(i) - \tilde p_{m+l+1}(i)> 0,
	\end{align*}
	and particularly, $\mathbb P\left( \tilde X_{l+1}^{N,m} = i \right) = \tilde p_{m+l+1}(i)$.
	
	Using the optional stopping theorem for the martingale $(\tilde B^m_k)_{k = -m}^\infty$ allows us to compute 
	\begin{multline*}
		\mathbb P\left( \tilde X^{N,m}_{l+1} = j_1 \mid \tilde X^{N,m}_l = i, T^{N,m} = \frac{l+1}{m}\right)
	\\	= \mathbb P\left( \inf\{k \in \N \colon \tilde B^m_{k - m} = j_1 - i \} < \inf\{k \in \N \colon \tilde B_{k-m}^m = j_2 - i \}\right) = \frac{j_2 - i}{j_2 - j_1}.
	\end{multline*}
	Thus, we obtain by the discrete version of the heat equation \eqref{discrete heat equation} and the inductive assumption
	\begin{gather}
		\nonumber
		\mathbb P\left( \tilde X^{N,m}_{l+1} = j_1, \tilde X_l^{N,m} \in C^{N,m}_i \right) = \sum_{k = j_1+1}^{j_2 - 1} \frac{j_2-k}{j_2 - j_1} \left( \mathbb P\left( \tilde X_l^{N,m} = k \right) - \mathbb P \left( \tilde X_{l+1}^{N,m} = k \right) \right)
	\\	 \nonumber 
	= \sum_{k = j_1+1}^{j_2 - 1} \frac{j_2-i}{j_2 - j_1} \left( \frac{1}{2} \tilde p_{m+l}(k) - \frac{1}{4} \tilde p_{m+l}(k-1) - \frac{1}{4} \tilde p_{m+l}(k+1) \right)
	\\	\label{outflow from C1}
	= \frac{1}{4}\left(  \tilde p_{m+l}(j_1 + 1) - \tilde p_{m+l}(j_1) + \frac{\tilde p_{m+l}(j_1) - \tilde p_{m+l}(j_2)}{j_2 - j_1} \right).
	\end{gather}
	Analogously, we get
	\begin{equation} \label{outflow from C2}
		\mathbb P\left( \tilde X_{l+1}^{N,m} = j_2, \tilde X_l^{N,m} \in C_i^{N,m}\right) = \frac{1}{4} \left( \tilde p_{m+l}(j_2 - 1) - \tilde p_{m+l}(j_2) - \frac{\tilde p_{m+l}(j_1) - \tilde p_{m+l}(j_2)}{j_2 - j_1}  \right)
	\end{equation}
	
	\item Finally, let $i \in \partial G^{N,m}$ and denote by $j_1, j_2 \in G^{N,m}$ the neighbours in $G^{N,m}$ of $i$ with $j_1 < i < j_2$.
	That means, $]j_1,i[ \cap \Z \, \subseteq C^{N,m}$ and $]i,j_2[ \, \cap \Z \subseteq C^{N,m}$.
	The optimal stopping theorem yields
	\begin{align} \label{flow G1}
		\mathbb P\left( \tilde X_{n+1}^{N,m} = j_1 \mid \tilde X_{n}^{N,m} = i \right) &= \frac{1}{4(i - j_1)} = \mathbb P\left( \tilde X_{n+1}^{N,m} = i \mid \tilde X_{n}^{N,m} = j_1 \right),
	\\	\label{flow G2} \mathbb P\left( \tilde X_{n+1}^{N,m} = j_2 \mid \tilde X_{n}^{N,m} = i \right) &= \frac{1}{4(j_2 - i)} = \mathbb P\left( \tilde X_{n+1}^{N,m} = i \mid \tilde X_{n}^{N,m} = j_2\right).
	\end{align}
	By \eqref{outflow from C1}, \eqref{outflow from C2}, \eqref{flow G1}, \eqref{flow G2}, and the inductive assumption, we find that the inflows and outflows are precisely canceling out such that $\mathbb P\left( \tilde X^{N,m}_{l+1} = i \right) = \tilde p_{m+l+1}(i)$, which concludes the inductive step.
	\end{enumerate}
\end{proof}

The next lemma ensures that the discrete version of inequality \eqref{lazy mass change} holds, which we use to prove that the process $\tilde X^{N,m}$ jumps with the correct rate from $C^N$ to $\partial G^{N,m}$.

\begin{lemma} \label{lem:discrete probabilities decreasing}
	Let $l,j \in \N$, $l \geq 2 j^2$. Then
	\begin{equation}
		\label{discrete probabilities decreasing}
		\tilde p_l(j) = \mathbb P\left(\sum_{k = 1}^{l} \zeta_k = j \right) > \mathbb P\left( \sum_{k = 1}^{l+1} \zeta_k = j \right) = \tilde p_{l+1}(j).
	\end{equation}
\end{lemma}

\begin{proof}
	The probabilties in \eqref{discrete probabilities decreasing} have a closed form given by \eqref{discrete normal distribution}.
	We compute the ratio of left-hand side to right-hand side of \eqref{discrete probabilities decreasing}:
	\begin{align*}
		\frac{\mathbb P\left(\sum_{k = 1}^{n} \zeta_k = j \right)}{ \mathbb P\left( \sum_{k = 1}^{n+1} \zeta_k = j \right)} &= \frac{2^{-2n}}{2^{-2(n+1)}} \frac{\binom{2n}{n+j}}{\binom{2(n+1)}{n+1+j}}
		= \frac{(n+1+j)(n+1-j)}{(n+1)(n+2)}
	\\	&= 1 + \frac{\frac{n+1}{2} - j^2}{(n+1)(n+2)} > 1,
	\end{align*}
	where the last inequality holds by assumption, and conclude with \eqref{discrete probabilities decreasing}.
\end{proof}

\subsubsection{Convergence of \texorpdfstring{$X^{N,m}$}{TEXT} to \texorpdfstring{$X^N$}{TEXT}}

In the last subsection we have shown that $X^{N,m}$ has the correct marginals.
It remains to prove the appropriate convergence of $X^{N,m}$ to $X^N$, which is the purpose of this subsection.

\begin{theorem}
	\label{XNm to XN}
	As random variables on $\mathcal D([0,\infty[)$, $(X^{N,m})_{m \in \N}$ converges weakly to $X^N$.
\end{theorem}

Before proving this theorem we prepare useful ingredients which we use in its proof. 

\begin{proposition} \label{BmTNm to BTN}
	As random variables on $\mathcal D([0,\infty[) \times [0,\infty[$, $(B^m,T^{N,m})_{m \in \N}$ converges weakly to $(B,T^N)$.
\end{proposition}

\begin{proof}
	To make the proof rigorous, we choose a 1-bounded metric $d$ compatible with the topology on $\mathcal D([0,\infty[)$ with the property that
	\begin{equation}
		\label{Skorokhod metric time shift}
		d(f,(x + f((t-s)\wedge 0))_{t \geq 0}) \leq s + |x| ,\quad (x,s) \in \R \times [0,\infty[.
	\end{equation}
	For $\mu,\nu \in \mathcal P(\mathcal D([0,\infty[))$ and $\tilde\mu,\tilde\nu \in \mathcal P(\mathcal D([0,\infty[) \times [0,\infty[)$ we denote their 1-Wasserstein distance by
	\begin{align}
		\label{Skorokhod space Wasserstein distance}
		\begin{split}
		\mathcal W(\mu,\nu) :=& \inf \left\{ \mathbb E\left[ d(X,Y) \right] \colon X \sim \mu, Y \sim \nu\right\}
	\\	\tilde{\mathcal W}(\tilde \mu,\tilde \nu) :=& \inf\left\{ \mathbb E \left[ d(X,Y) + |S-T| \right] \colon (X,S) \sim \tilde \mu, (Y,T) \sim \tilde \nu \right\}.
		\end{split}
	\end{align}
	We write $\mu^{N,m}$, $\mu^N$, $\mu^m$, and $\mu$ for the laws of $(B^m,T^{N,m})$, $(B,T^N)$, $(\sqrt{\frac{2}{m}}\sum_{k = 1}^{\lfloor mt \rfloor}\zeta_j )_{0 \leq t < \infty}$, and a standard Brownian motion $\tilde B$ starting at $0$, respectively, on $\mathcal D([0,\infty[) \times [0,\infty[$.

	For fixed $h > 0$ consider a spatial partition $(P_{j})_{j \in \N}$ of $\R$ into intervals of maximal length $h$ with $I_1 \cup I_2 = \N$ and
	\[
		\bigcup_{j \in I_1} P_j = G^N, \quad \bigcup_{j \in I_2} P_j = C^N.
	\]
	Recall that the marginal distributions of $B^m$ and $B$ are converging, thus, Portmanteau's theorem yields
	\[
		\sum_{x \in P_j} p^m(t,x) = \mathbb P\left( B_t^m \in P_j \right) \overset{m \to \infty}{\longrightarrow} \mathbb P\left( B_t \in P_j \right).
	\]
	Additionally we know by Lemma \ref{lem:discrete probabilities decreasing} that for any $x\in C^N \subseteq [-1,1]$ the probabilties $p^m(t,x)$ are strictly decreasing in $t \in [0,\infty[$, therefore, if $j \in I_2$ and $k \in \N$ then
	\begin{align*}
		&\mathbb P\left( B_0^m \in P_j, kh < T^{N,m} \leq (k+1)h \right)
	\\	& \hspace{1cm} = \mathbb P\left( B_0^m \in P_j, \frac{p((k+1)h,B_0^m)}{p(0,B_0^m)} \leq U < \frac{p(kh,B_0^m)}{p(0,B_0^m)}\right)
	\\	& \hspace{1cm} = \sum_{x \in P_j} p(1+kh,x) - p(1 + (k+1)h, x)
		\overset{m \to \infty}{\longrightarrow} \mathbb P\left( B_{kh} \in P_j \right) - \mathbb P\left( B_{(k+1)h} \in P_j \right)
	\\	& \hspace{8.8cm} = \mathbb P\left( B_0\in P_j, kh < T^N \leq (k+1)h \right).
	\end{align*}
	We introduce the real-valued sequences
	\begin{align*}	
		a_j^m :=& \mathbb P\left( B_0^m \in P_j \right) \vee \mathbb P\left( B_0 \in P_j \right),
	\\	b_j^m :=& \mathbb P\left( B_0^m \in P_j \right) \wedge \mathbb P\left( B_0 \in P_j \right),
	\\	a^m_{j,k} :=& \mathbb P\left( B_0^m \in P_j, kh \leq T^{N,m} < (k+1)h \right) \vee \mathbb P\left( B_0 \in P_j, kh \leq T^N < (k+1)h \right),
	\\	b^m_{j,k} :=& \mathbb P\left( B_0^m \in P_j, kh \leq T^{N,m} < (k+1)h \right) \wedge \mathbb P\left( B_0 \in P_j, kh \leq T^N < (k+1)h \right),
	\end{align*}
	for $j,k \in \Z \times \N$,
	and note that by the previous arguments $(a^m_j - b^m_j)_{m \in \N}$ and $(a^m_{j,k} - b^m_{j,k})_{m \in \N}$ are vanishing sequences.
	Some analysis shows that
	\begin{equation}
		\label{null sequences}
		\lim_{m \to \infty}\sum_{j \in I_1} a^m_j - b^m_j = 0, \quad \lim_{m \to \infty} \sum_{j \in I_2} \sum_{k \in \N} a^m_{j,k} - b^m_{j,k} = 0.
	\end{equation}
	Moreover conditional on $B_0^m = x$ we have $(B_{t}^m)_{t \geq 0} \sim (\sum_{j = 1}^{\lfloor tm \rfloor} \zeta_j + x)_{t \geq 0}$.
	We estimate the the distance for particles starting in $P_j$ where either $j \in I_1$ or $j \in I_2$ separately:
	\begin{equation}
		\label{error bounded depending on h}
		\tilde{\mathcal W}\left( \mu^{N,m}, \mu^N \right) \leq  \mathcal W(\mu^m,\mu) + 2h + \sum_{j \in I_1} a^m_j - b^m_j + \sum_{j \in I_2} \sum_{k \in \N} a^m_{j,k} - b^m_{j,k}.
	\end{equation}
	On $j \in I_1$ we obtain due to the spatial partitions an error of maximal $h + \mathcal W_1(\mu^m,\mu)$ plus (as the metric $d$ is bounded by 1) the excess mass $\sum_{j \in I_1} a^m_j - b^m_j$.
	For $j \in I_2$ we have to take into account an additional error (of maximal $h$) due to the second components of $(B^m,T^{N,m})$ and $(B,T^N)$.
	From \eqref{null sequences} and \eqref{error bounded depending on h} it is easily deductible that
	\[
		\limsup_{m \to \infty} \tilde{ \mathcal W}_1\left( \mu^{N,m}, \mu \right) \leq 2h,
	\]
	which yields the desired convergence as $h > 0$ was arbitrary.
\end{proof}

\begin{lemma}
	\label{lem:time changed path converges}
	Let $f \in C([0,\infty[)$ be a path satisfying \eqref{path property a} and \eqref{path property b}.
	Let $(f^n)_{n \in \N}$ be a sequence converging to $f$ in $\mathcal D([0,\infty[)$ with $\lim_{t \to \infty}\tau_t^N(f^n) = \infty$, where
	the time change $\tau^N$ is given by \eqref{tau^N}.
	Then the time changed paths satisfy
	\begin{equation}
		\label{time changed path converges}
		(f^n(\tau^N_t(f^n)))_{0 \leq t < \infty} \to (f(\tau^N_t(f)))_{0 \leq t < \infty} \text{ in }\mathcal D([0,\infty[).
	\end{equation}
\end{lemma}

\begin{proof}
	Since $f$ is continuous, the convergence of $(f^n)_{n \in \N}$ to $f$ in the Skorokhod space $\mathcal D([0,\infty[)$ is  by \cite[Chapter VI. Proposition 1.17]{JaSh13} equivalent to locally uniform convergence.
	Moreover, the time changes $t \mapsto \tau^N_t(f)$ and $t \mapsto \tau^N_t(f^n)$ are strictly increasing, right-continuous functions with
	\begin{equation}
		\label{jumps of time change}
		\Delta \tau_t^N(f) :=   \tau_t^N(f) - \tau_{t-}^N(f)  > 0 \iff A^N_{\tau_{t-}^N(f)}(f) = A^N_{\tau_t^N(f)}(f).
	\end{equation}
	
	As $f$ satisfies \eqref{path property b} and $(f^n)_{n \in \N}$ converges locally uniformly to $f$, we have for $\lambda$-almost every $s \in [0,\infty[$ that
	\[
		\lim_{n \to \infty} \mathds 1_{G^N}(f^n(s)) = \mathds 1_{G^N \setminus \partial G^N}(f(s)) = \mathds 1_{G^N}(f(s)).
	\]
	Particularly, the convergence of $(A^N_s(f^n))_{n \in \N}$ to $A^N_s(f)$ holds pointwise and for any $t \in [0,\infty[$
	\begin{equation}
		\label{time change liminf}
		\limsup_{n \to \infty} \tau_t^N(f^n) \leq \tau_t^N(f).
	\end{equation}

	First we show that
	\begin{equation} \label{discrete time change converges 1}
		\Delta \tau_t^N(f) = 0 \implies \lim_{n\to \infty} \tau_t^N(f^n) = \tau_t^N(f).
	\end{equation}
	If $\Delta \tau_t^N(f) = 0$, we have for any sequence $t_k \nearrow t$ that $\lim_{k \to \infty} \tau_{t_k}^N(f) = \tau_t^N(f)$.
	The map $t \mapsto \tau^N_t(f)$ is strictly increasing, therefore there are $s_k \nearrow \tau_t^N(f)$ with  $\tau_{t_k}^N(f) < s_k < \tau_{t}^N(f)$ and $t_k < A_{s_k}^N(f) < t_{k+1}$.
	We find for any $k \in \N$
	\begin{equation} \label{time change limsup}
		\limsup_{n \to \infty} \tau_{t}^N(f^n) \geq \liminf_{n \to \infty} \tau_{t_{k+1}}^N(f^n) \geq  s_k.
	\end{equation}
	Combining \eqref{time change liminf} with \eqref{time change limsup} yields \eqref{discrete time change converges 1}.
	Moreover, we get
	\begin{align*}
		\limsup_{n \to \infty} \Delta \tau_t^N(f^n) &\leq \limsup_{n \to \infty} \tau_t(f^n) - \tau_{t_{k+1}}^N(f^n)
	\\	&= \tau_t^N(f) - \liminf_{n \to \infty} \tau_{t_{k+1}}^N(f^n) \leq \tau_t^N(f) - s_k,
	\end{align*}
	which yields that 
	\begin{equation}
		\label{time change Delta}
		\Delta \tau_t^N(f) = 0 \implies \lim_{n \to \infty} \Delta \tau_t^N(f^n) =  \Delta \tau_t^N(f).
	\end{equation}

	As $\tau_t^N(f)$ is increasing, it can consist of at most countably many jumps.
	Therefore, we have that $(\tau_t^N(f^n))_{n \in \N}$ has $\tau_t^N(f)$ as limit for a dense subset of $[0,\infty[$. 
	By \cite[Chapter VI. Lemma 2.25]{JaSh13} it remains to verify that for any $t \geq 0$ there is a sequence $(t_n)_{n \in \N}$ with
	\begin{enumerate}[label = (\roman*)]
		\item \label{it:discrete time change converges 1} $t_n \to t$, and $t_n \leq t$ if $t$ is continuity point of $\tau^N(f)$,
		\item \label{it:discrete time change converges 2} $\Delta\tau_{t_n}^N(f^n) \to \Delta\tau_t^N(f)$.
	\end{enumerate}

	So assume that $\Delta \tau_t^N(f) > 0$.
	Let $s_0 := \tau_{t-}^N(f)$, $s_1 := \tau_t^N(f)$, and $\bar{s} := \frac{s_0 + s_1}{2}$.
	Let $\epsilon > 0$ be fixed. As $f$ satisfies \eqref{path property a}, there are $s_0^\epsilon$, $s_1^\epsilon$
	with $f(s_0^\epsilon), f(s_1^\epsilon) \in G^N \setminus \partial G^N$ and
	\[
		s_0 - \epsilon \leq s_0^\epsilon < s_0 < s_1 < s_1^\epsilon \leq s_1 + \epsilon.	
	\]
	Since $f^n \to f$ locally uniformly, $f$ is continuous, and $C^N$ is the union of open intervals, there are $s_0^n, s_1^n$ with $f^n(s) \in C^N$, $s \in [s_0^n,s_1^n]$ and $s_i^n \to s_i, i \in \{0,1\}$.
	Write $\bar{s}^n := \frac{s_0^n + s_1^n}{2}$, and note that $s_1^n - s_0^n \leq \Delta \tau_{A_{\bar{s}^n}}^N(f^n)$, whence,
	\begin{equation} \label{delta liminf}
		\Delta \tau_t^N(f) = \lim_{n \to \infty} s_1^n - s_0^n \leq \liminf_{n \to \infty} \Delta \tau_{A_{\bar{s}^n}}^N(f^n).
	\end{equation}
	For $n$ sufficiently large, we have that $f^n(s_0^\epsilon), f^n(s_1^\epsilon) \in G^N \setminus \partial G^N$.
	Since $f^n$ is a c\`adl\`ag path, it spends a positive amount of time during $[s_0^\epsilon,s_0^n[$ and $]s_1^n,s_1^\epsilon + \epsilon]$ in $G^N$.
	Therefore, when $n$ is sufficiently large we have
	\begin{equation}
		\label{delta limsup}
		\Delta \tau_{A_{\bar{s}^n}}^N(f^n) \leq s_1^\epsilon - s_0^\epsilon + \epsilon \leq \Delta \tau_t^N(f) + 3\epsilon.
	\end{equation}
	By \eqref{delta liminf} and \eqref{delta limsup} we have that $(\Delta \tau^N_{A_{\bar{s}^n}}(f^n))_{n \in \N}$ converges to $\Delta \tau^N_t(f)$.
	
	The convergence of $A^N(f^n)$ to $A^N(f)$ is even locally uniform, therefore $A^N_{\bar{s}^n}(f^n) \to A^N_{\bar{s}}(f) = A^N_{\tau_t^N(f)}(f) = t$.
	We conclude that setting $t_n := A^N_{\bar{s}^n}(f^n)$ provides the desired sequence $(t_n)_{n \in \N}$ satisfying item \ref{it:discrete time change converges 1} and \ref{it:discrete time change converges 2}.
	We have shown that $\tau^N(f^n) \to \tau^N(f)$ in $\mathcal D([0,\infty[)$.

	By \cite[Chapter VI. Theorem 1.14]{JaSh13} there are continuous, bijections $(\alpha_n)_{n \in \N}$ on $[0,\infty[$ such that
	\begin{enumerate}[label = (\roman*)]
		\item $\sup_{0 \leq s < \infty} |\alpha_n(s) - s| \to 0$,
		\item $\sup_{0 \leq s \leq T} |\tau_{\alpha_n(s)}^N(f^n) - \tau_{s}^N(f)| \to 0$ for all $T>0$.
	\end{enumerate}
	Denote by $w_f^T$ the modulus of continuity of $f$ restricted to $[0,T]$.
	We have
	\begin{align*}
		\sup_{0 \leq s \leq T} &\left| f^n(\tau^N_{\alpha_n(s)}(f^n)) - f(\tau^N_{s}(f))\right| 
	\\	&\leq \sup_{0 \leq s \leq \tau^N_{\alpha_n(T)}(f^n)} \left| f^n(s) - f(s) \right| + \sup_{0 \leq s \leq T} \left| f(\tau_{\alpha_n(s)}^N(f^n)) - f(\tau^N_s(f)) \right|
	\\	&\leq \sup_{0 \leq s \leq \tau^N_{\alpha_n(T)}(f^n)} \left| f^n(s) - f(s) \right| + w_f^T\left(\sup_{0 \leq s \leq T} \left| \tau_{\alpha_n(s)}^N(f^n) - \tau^N_s(f) \right|\right),
	\end{align*}
	which vanishes uniformly for $n \to \infty$.
	Applying once again \cite[Chapter VI. Theorem 1.14]{JaSh13} establishes the assertion.
\end{proof}

\begin{lemma} \label{lem:path properties}
	The original Brownian motion $B$ puts full mass on paths having the property that
\begin{align}
	\label{path property a}
	B_t(\omega) \in \partial G^N & \implies \forall \epsilon > 0\, \exists s \in [t-\epsilon,t+\epsilon] \text{ s.t. } B_s(\omega) \in G^N \setminus \partial G^N,
	\\	\label{path property b}
		\lambda& \left(t \in [0,\infty[ \colon B_t(\omega) \in \partial G^N \right) = 0.
\end{align}
\end{lemma}

\begin{proof}
	Property \eqref{path property b} is evident from the occupation formula and the Brownian local time.
	The boundary $\partial G^N$ of $G^N$ consists only of finitely many points.
	Thus, to see \eqref{path property a}, it suffices to show for fixed $\epsilon > 0$ that
	\begin{equation}
		\label{path prop to show}
		\mathbb P\left( \left\{B_t = 0 \implies \exists s \in [t-\epsilon,t+\epsilon] \colon B_s > 0 \right\} \right) = 1.	
	\end{equation}
	Consider the countably, increasing family of stopping times
	\[
		\tau^k_\epsilon(\omega) := \inf\left\{ s > k\frac{\epsilon}{2} \colon B_s = 0\right\},
	\]
	where $k \in \N$.
	We have that $B_t(\omega) = 0$ if and only if there is $k \in \N$ with $|\tau^k_\epsilon(\omega) - t| < \frac{\epsilon}{2}$.
	From here we deduce
	\begin{multline} \label{path prop 1}
		\Big\{ \forall t \in [0,\infty[ \text{ with } B_t = 0 \, \exists s \in [t-\epsilon,t+\epsilon] \text{ s.t. } B_s > 0 \Big\} \\ \supseteq
		\left\{ \forall k \in \N \, \exists s \in \left[\tau^k_\epsilon- \frac{\epsilon}{2}, \tau^k_\epsilon + \frac{\epsilon}{2}\right]\text{ with } B_s > 0 \right\}.
	\end{multline}
	Due to the strong Markov property and Blumenthal's 0-1 law we obtain for any $k \in \N$ that
	\begin{multline} \label{path prop 2}
		\mathbb P\left(\exists s \in \left[\tau^k_\epsilon - \frac{\epsilon}{2}, \tau^k_\epsilon + \frac{\epsilon}{2} \right] \text{ with } B_s > 0 \right) \geq 
		\mathbb P\left(\exists s \in \left[\tau^k_\epsilon, \tau^k_\epsilon + \frac{\epsilon}{2}\right] \text{ with } B_s > 0 \right)
	\\	\geq \mathbb P\left( \exists s \in \left[0,\frac{\epsilon}{2}\right] \text{ with } B_s > 0 \mid B_0 = 0\right) = 1.
	\end{multline}
	By \eqref{path prop 2} and \eqref{path prop 1} we get \eqref{path prop to show}, which concludes the proof.
\end{proof}

\begin{proof}[Proof of Theorem \ref{XNm to XN}]
	By Proposition \ref{BmTNm to BTN} we have convergence in law of $(B^{m},T^{N,m})_{m \in \N}$ to $(B,T^N)$.
	By the Skorokhod representation theorem we may assume w.lo.g.\ that this convergence holds almost surely.
	Due to Lemma \ref{lem:path properties}, we can apply Lemma \ref{lem:time changed path converges}, and find that almost surely
	\begin{equation}
		\left(B_0^m, (B_{\tau^{N,m}_t}^m)_{0 \leq t < \infty},T^{N,m} \right) \to \left(B_0, (B_{\tau^N_t})_{0 \leq t < \infty}, T^N \right)
	\end{equation}
	in $\R \times \mathcal D([0,\infty[) \times [0,\infty[$, when $m \to \infty$.
	It is readily verified that on $\{T^N > 0\} \cup \{ B_0 \in G^N \setminus \partial G^N \}$ we have almost surely that $(X^{N,m})_{m \in \N}$ has $X^N$ as its limit (in $\mathcal D([0,\infty[)$).
	Moreover, $\{T^N > 0\} \cup \{ B_0 \in G^N \setminus \partial G^N \}$ has full probability, hence, we conclude with the assertion.
\end{proof}

\subsection{Proof of Theorem \ref{continuous} and Proposition \ref{finite}}

\begin{proof}[Proof of Proposition \ref{finite}]
	The strong Markov property w.r.t.\ the right-continuous version of its natural filtration $\mathcal F^N$ is readily derived from Lemma \ref{BtauN is Feller}.
	Lemma \ref{lem:path properties} tells us that almost surely for all $t \in [0,\infty[$ the time change $\tau_t < \infty$, thus, $X^N$ is well-defined.
	The trajectories $(\tau_t^N(\omega))_{0 \leq t < \infty}$ are increasing and c\`adl\`ag, whence, $(X_t^N(\omega))_{0 \leq t < \infty}$ is also c\`adl\`ag.

	It remains to convince ourselves that $X^N$ is a martingale with the correct Brownian mar\-ginals.
	By Theorem \ref{XNm and Bm same 1dim dist} we have for any $t \in [0,\infty[$ that $X^{N,m}_t \sim B^m_t$.
	Theorem \ref{XNm to XN} provides weak convergence of $(X^{N,m})_{m \in \N}$ to $X^N$.
	The laws of $(X^{N,m}_t)_{m \in \N}$ converge to $X^N_t$, hence, $X^N_t \sim B_t$.
	The following computation establishes that the second moments of $X^{N,m}_t$ are converging to $t$:
	\begin{align*}
		\mathbb E\left[ (B^m)^2 \right] &= \frac{2}{m^2} \sum_{(i_1,i_2) \in \{1,\ldots, \lfloor m t \rfloor\}^2}\mathbb E\left[\zeta_{i_1}\zeta_{i_2}\right]
		= \frac{1}{m} \#\left\{ (i_1,i_2) \in \{1,\ldots,\lfloor m t \rfloor \}^2 \colon i_1 = i_2 \right\} = \frac{\lfloor m t \rfloor}{m},
	\end{align*}
	when $m \in \N$.
	Therefore, the second moments of $(X^{N,m}_t)_{m \in \N}$ have the second moments of $X^N_t$ as their limit.
	By \cite[Theorem 6.9]{Vi09} the joint distributions $(X_t^{N,m}, X_s^{N,m})_{m \in \N}$, which are martingale couplings when $s \leq t$, converge for all pairs $(s,t) \in [0,\infty[^2$ in 2-Wasserstein distance to the law of $(X_t^N, X_s^N)$.
	Thus, the law of $(X_t^N, X_s^N)$ constitutes a martingale coupling.
	By the Markov property of $X^N$, we find that $X^N$ is a martingale as almost surely
	\[
		\mathbb E\left[ X_s^N \mid \mathcal F_t^N \right] = \mathbb E\left[ X_s^N \mid X_t^N \right] = X_t^N,
	\]
	when $(\mathcal F_t^N)_{0 \leq t < \infty}$ denotes the right-continuous version of the natural filtration of $X^N$.
\end{proof}

\begin{proof}[Proof of Theorem \ref{continuous}]
	The trajectories $(A_t^N(\omega))_{0 \leq t < \infty}$ increase pointwise to the almost surely strictly increasing process $(A_t(\omega))_{0 \leq t < \infty}$.
	Consequentially, the time changes $(\tau_t^N(\omega))_{0 \leq t < \infty}$ converge almost surely pointwise to $(\tau_t(\omega))_{0 \leq t < \infty}$. Hence, there holds almost surely
	\begin{equation}
		(B_{\tau^N_t})_{0 \leq t < \infty} \to (B_{\tau_t})_{0 \leq t < \infty} \text{ in }\mathcal D([0,\infty[).
	\end{equation}
	From here, it is evident that $(X^N)_{N \in \N}$ has $X$ as its weak limit.
	By Proposition \ref{finite}, we get that $X$ has Brownian marginals.
	At each time $t \in [0,\infty[$ the contribution of the ``busy'' particles to the mass on $C$ vanishes, i.e.\ $\mathbb P\left( X_t \in C, T \leq t \right) = 0$.
	Therefore we can consider the transition kernel of $X$ for $X_t \in C$ and $X_t \in G$ seperately.
	For $x \in G$ the transition kernel coincides with the one of the Feller process $(B_{\tau_t})_{0 \leq t < \infty}$ (conditionally on $B_0 = x$).
	For $x \in C$ the transition kernel is given by a mix of a dirac at $x$ plus an appropriate convolution in time of the kernel of  $(B_{\tau_t})_{0 \leq t < \infty}$ (conditionally on $B_0 = x$).
	Altogether, it is possible to explicitly write down the Markov kernel corresponding to $X$, thus, $X$ is a Markov process.

	The martingale property follows by the same argument as in Proposition \ref{finite}.	
\end{proof}

\begin{remark}\label{exp} We have formulated the ``faking'' procedure for the case of a standard Brownian motion. After all, this is the most regular and canonical situation one can imagine. But our construction applies to general martingales of the form

$$ dX_t = \sigma(t,X_t) dB_t,$$
for arbitrary (sufficiently regular) $\sigma(\cdot,\cdot)$.
We only demonstrate this at the hand of the example of an exponential Brownian motion, i.e., for $\sigma(x,t) = x$, so that

$$ X_t = \exp\left(B_t -\frac{t}{2}\right),\qquad t \ge 0, $$
where $B_t$ is a standard Brownian motion, starting at $0$. To produce a ``fake'' exponential Brownian motion, find $0 < t_1 < t_2 < \infty$ and $0<a<b<\infty$ such that, for each $t_1 \le t\le t_2$, the function $x \to xp(x,t)$ is concave so that the density function function $p(t,x)$ is decreasing in view of the heat equation

$$ \frac{\partial}{\partial t} p(t,x) = \frac{\partial^2}{\partial x^2} (xp(t,x)). $$
Now play the same game as we did for Brownian motian $B_t$ on the interval $]0,1[$ on the interval $]a,b[$, i.e.\ choosing a sequence of disjoint intervals $([a_n,b_n])_{n=1}^{\infty}$ whose union is dense in $[a,b]$ but of strictly smaller Lebesgue measure than $b-a$ etc etc. The above construction carries over verbatim. In view of the negativity of $ \frac{\partial}{\partial t} p(t,x)$ we again have a positive rate of the transition from the status of ``lazy'' to ``busy'' particles so that nothing has to be changed.

\end{remark}

\section{An alternative construction based on \cite[Theorem 1.3]{Lo08b}.}\label{GLSection}

In this final section, we present an alternative construction of a continuous Markovian fake Brownian motion based on the following theorem from \cite[Theorem 1.3]{Lo08b}.

\begin{theorem}\label{LowThm}
Let $(\mu_t)_{t\geq 0}$ be a family of probability measures on $\R$, increasing in convex order, and assume that $t\mapsto \mu_t$ is weakly continuous and that each $\mu_t$ has convex support. Then there exists a unique strongly continuous Markov martingale $X$ such that $X_t\sim \mu_t, t\geq 0$. 
\end{theorem}
To start the construction of a fake Brownian motion, choose  a symmetric Borel set $A$ in $\R$ such that $A$ as well as its complement $B := \R \setminus A$ are dense in $\R$ and have positive Lebesgue measure. Let $p$ be the probability that a standard normal lies in $A$ and $q = 1-p$ that it lies in $B$. 

Let $U$ be a standard normal random variable conditional on being in $A$, and $V$ a standard normal conditional on being in $B$. By symmetry of the distributions, we know that  the laws of $t^{\frac{1}{2}} U$ and $t^{\frac{1}{2}} V$ are increasing in the convex order. Indeed, we may write the laws of $U$ and $V$ as integrals over laws of the form $\frac{1}{2} (\delta_x + \delta_{-x})$. Multiplying $x$ with the increasing function $t \mapsto t^{\frac{1}{2}}$ these measures are clearly increasing in convex order. By integrating over these measures we get the same assertion for the laws of $t^{\frac{1}{2}} U$ and $t^{\frac{1}{2}} V$.

Now we are in a position to apply Theorem \ref{LowThm}: there exist continuous strong Markov martingales $Y$ and $Z$ such that $Y_t$ is distributed as $t^{\frac{1}{2}} U$ and $Z_t$ as $t^{\frac{1}{2}} V$, for all $t \ge 0$.

Finally let $w$ be a Bernoulli random variable equal to $1$ with probability $p$ and independent of $Y$ and $Z$. Set $X_t = Y_t$ if $w=1$, and $X_t = Z_t$ if $w=0$.

This process $X$ will be Markov. Indeed at every time $t > 0$ the random variable $X_t$ lies either in $t^{\frac{1}{2}} A$ or in its complement. Depending on this information, the process thus follows the regime of $Y$ or $Z$. In particular, this also shows that the process $X$ is certainly not a Brownian motion. It is also clear that the marginal distributions of the process $X$ are the corresponding Gaussian laws of a Brownian motion.

This concludes the construction.

\bibliographystyle{abbrv}
%\bibliography{joint_biblio, ../FilteredProcesses/lib/joint_biblio}
%\bibliography{../FilteredProcesses/lib/joint_biblio}
\bibliography{../MBjointbib/joint_biblio}

\end{document}